\documentclass[a4paper,11pt]{article}
\usepackage{fullpage}
\usepackage{amsmath,amsthm,amsfonts,amssymb}
\usepackage{enumerate}
\usepackage{dsfont}
\usepackage{natbib}%
\usepackage{graphicx}
\usepackage{changes}
\usepackage{color}
\usepackage{hyperref} 
\hypersetup{
colorlinks=true,
breaklinks=true, 
urlcolor= blue,
linkcolor= blue,
citecolor=blue,
}

\newtheorem{theorem}{Theorem}

\newtheorem{proposition}[theorem]{Proposition}
\newtheorem{lemma}[theorem]{Lemma}

\newtheorem{definition}{Definition}

\newtheorem{remark}{Remark}

\DeclareMathOperator{\supp}{supp}

\begin{document}

\title{Rademacher complexity for Markov chains : Applications to kernel smoothing and Metropolis-Hasting}
\author{Patrice Bertail, Modal'X, UPL, Universit\'{e} Paris-Nanterre\\
\ \\
 Fran\c{c}ois Portier, T\'el\'ecom ParisTech, Universit\'e Paris-Saclay}


\maketitle

\begin{abstract}

Following the seminal approach by Talagrand, the concept of Rademacher complexity for independent sequences of random variables is extended to Markov chains. The proposed notion of ``block Rademacher complexity'' (of a class of functions) follows from renewal theory and allows to control the expected values of suprema (over the class of functions) of empirical processes based on Harris Markov chains as well as the excess probability. For classes of Vapnik-Chervonenkis type, bounds on the ``block Rademacher complexity'' are established. These bounds depend essentially on the sample size and the probability tails of the regeneration times. The proposed approach is employed to obtain convergence rates for the kernel density estimator of the stationary measure and to derive concentration inequalities for the Metropolis-Hasting algorithm.

\end{abstract}

\noindent \textbf{Keywords} : Markov chains; Concentration inequalities; Rademacher complexity; Kernel smoothing; Metropolis Hasting.


\section{Introduction}

Let $(\Omega, \mathcal F, \mathbb P)$ be a probability space and suppose that $X = (X_i)_{i\in \mathbb N}$ is a sequence of random variables on $(\Omega, \mathcal F, \mathbb P)$ valued in $(E,\mathcal E)$.   Let $\mathcal F$ denote a countable class of real-valued measurable functions defined on $E$. Let $n\in \mathbb N$, define
\begin{align*}
Z =  \sup_{f\in \mathcal F}\left| \sum_{i=1}^n (f(X_i) - \mathbb E [f(X_i)] )  \right| .
\end{align*} 
The random variable $Z$ plays a crucial role in machine learning and statistics: it can be used to bound the risk of an algorithm \citep{vapnik:1998} as well as to study \textit{M and Z estimators} \citep{vandervaart:1998}; it serves to describe the (uniform) accuracy of function estimates such as the cumulative distribution function \citep{shorack+w:2009,boucheron+l+m:2013} or \textit{kernel smoothing estimates} of the probability density function \citep{einmas0,gine+g:02}; in addition, kernel density estimators, as well as their variations,
\textit{Nadaraya-Watson} estimators, are at the core of many
semi-parametric statistical procedures \citep{akritas+v:2001,portier:2016} in which controlling $Z$-type quantities permits to take advantage of the tightness of the empirical process \citep{van+w:2007}.
Depending on the class $\mathcal F$ many different bounds are known when $X$ forms an independent and identically distributed (i.i.d.) sequence of random variables. A complete picture is given in \cite{wellner1996,boucheron+l+m:2013}.

The purpose of the paper is to study the behavior of $Z$ when $X$ is a Markov chain. The approach taken in this paper is based on renewal theory and is known as the
\textit{regenerative method}, see \cite{Smith55,nummelin:1978,athreya+n:1978}. Indeed it is well known that sample paths
of a Harris chain may be divided into i.i.d. \textit{regeneration
blocks}. These blocks are defined as data cycles between random times called \textit{regeneration times} at which the
chain forgets its past. Hence, many
results established in the i.i.d. setup may be extended to the Markovian
framework by applying the latter to (functionals of) the regeneration blocks. 
Refer to \cite{meyn+t:2009} for the \textit{strong law of large numbers} and the
\textit{central limit theorem}, to \cite{levental:1988} for functional CLT, as well as \cite{Bolthausen80, Malinovskii87,
Malinovskii89, BerClemPTRF,bednorz+l+l:2008,douc+m:2008} for refinements of the central limit theorem and \cite{adamczak:2008,BerClemTPA,bertail+c:2017} for exponential type bounds.

Other works dealing with concentration inequalities for Markov chains include, among others, \cite{joulin+o:2010}, where a concentration inequality is proved under a curvature assumption; \cite{dedecker+g:2015}, where the technique of \textit{bounded differences} is employed to derive an Hoeffding-type inequality; \cite{wintenberger:2017} which extends the previous to the case of unbounded chains. 

We introduce a new notion of complexity that we call the \textit{block Rademacher complexity} which extends the classical Rademacher complexity associated for  independent sequences of random variables \citep{boucheron+l+m:2013} to Markov chains. As in the independent case, the \textit{block Rademacher complexity} is useful to bound the expected values of empirical processes (over some classes of functions) and intervenes as well to control the excess probability. Depending on the probability tails of the regeneration times, which are considered to be either exponential or polynomial, we derive bounds on the block Rademacher complexity of Vapnik-Chervonenkis (VC) types of classes. Interestingly, the obtained bounds bears resemblance to the ones provided in \cite{einmas0,gigui2001} (for independent $X$) as they depend on the variance of the underlying class of functions $\mathcal F$ allowing to take advantage of classes $\mathcal F$ having small fluctuations.

To demonstrate the usefulness and the generality of the proposed approach, we apply our results on $2$ different problems. The first one tackles uniform bounds for the \textit{kernel estimator} of the stationary density and illustrates how to handle particular classes of functions having variance that decreases with the sample size. The second problem deals with the popular \textit{Metropolis Hasting} algorithm which furnishes examples of Markov chains that fit our framework.

\paragraph{Kernel density estimator.} The asymptotic properties of kernel density estimators, based on independent
and identically distributed data, are well understood since the
seventies-eighties \citep{stute:1982}. However finite sample properties were
only studied in the beginning of the century when the studies of empirical processes over VC class have been proved to be powerful to handle such kernel density estimates \citep{einmas0,gine+g:02}. The functions class of interest in this problem is given by 
\begin{align*}
\mathcal K _n = \{ x\mapsto K((x-y)/h_n)\,:\, y\in\mathbb R^d  \},
\end{align*}
where $K:\mathbb R^d \to \mathbb R$ is called the kernel and $(h_n)_{n\in\mathbb N}$ is a positive sequence converging to $0$ called the bandwidth. Based on the property that $\mathcal K _n$ is included on some VC class \citep{nolan+p:1987},  some master results have been obtained
by \cite{einmas0,einmas,gigui2001,gine+g:02} who proved some concentration inequalities,
based on the seminal work of \cite{talagrand:1996}, allowing to establish
precisely the rate of uniform convergence of kernel density estimators. Kernel density estimates are particular because the variance of each element in $\mathcal K_n$ goes to $0$ as $n\to \infty$. This needs to be considered to derive accurate bounds, e.g., the one presented in \cite{gine+g:02}. The proposed approach takes care of this phenomenon as, under reasonable conditions, our bound for Markov chains scales at the same rate as the ones obtained in the independent case. Note that our results extend the ones in \cite{azais:2016} where under similar assumptions the consistency is established.

The study of this specific class of statistics for dependent data has only recently received special attention in the statistical
literature. To the best of our knowledge, uniform results are limited to the alpha and beta mixing cases when dependency occurs  \citep{Peligrad:1992,hansen:2008} by using \textit{coupling
techniques}. 

\paragraph{Metropolis-Hasting algorithm.}
 
Metropolis-Hasting (MH) algorithm is one of the state of the art method in computational statistics and is frequently used to compute Bayesian estimators \citep{robert:2004}. Theoretical results for MH are often deduced from the analysis of geometrically ergodic Markov chains as presented for instance in \cite{mengersen+t:1996,roberts+t:1996,jarner+h:2000,robros:2004,douc+m+r:2004}. Whereas many results on the asymptotic behavior of MH are known, e.g., central limit theorem or convergence in total variation, only few non-asymptotic results are available for such Markov chains; see for instance \citep{latuszynski+m+n:2013} where the estimation error is controlled \textit{via} a Rosenthal-type inequality. We consider the popular \textit{random walk MH}, which is at the heart of the adaptive MH version introduced in \cite{haario+s+t:2001}. Building upon the pioneer works \cite{roberts+t:1996,jarner+h:2000} where the geometric ergodicity is established for the random walk MH, we show that whenever the class is VC the expected values of the sum over $n$ points of the chain is bounded by $D\sqrt{n (1\vee \log(\log(n)))}$ where $D>0$ depends notably on the distribution of the regeneration times. By further applying this to the quantile function, we obtain a concentration inequality for Bayesian credible intervals.

\paragraph{Outline.}

The paper is organized as follows. In section \ref{sec:back}, the notations and
main assumptions are first set out. Conceptual background related to the
renewal properties of Harris chains and the regenerative method are also
briefly exposed.  In section \ref{sec:Rad_markov}, the notion of
block Rademacher complexity for Markov chains is introduced as well as the notion of VC classes. Section \ref{sec:main_results} provides the main result of the paper : a bound on the Rademacher compexity. Our
methodology is illustrated in section \ref{sec:applications} on kernel density estimation and MH. Technical proofs are postponed to the Appendix.

\section{Regenerative Markov chains}\label{sec:back}

\subsection{Basic definitions}

In this section, for seek of completeness we recall
the following important basic definitions and properties of regenerative
Markov chains. An interested reader may look into \cite{nummelin:1984} or \cite{meyn+t:2009} for
detailed survey of regeneration theory.

Consider an homogeneous Markov chain $X=(X_{n})_{n\in\mathbb{N}}$ on~a~countably generated state
space~$(E,\mathcal{E})$ with transition probability~$P(.,.),$ initial
probability~$\nu$. The assumption that
$\mathcal{E}$ is countably generated allows to avoid measurability
problems.  For any $n\geq 1$, let $P^{n}$ denote the $n$-th iterate of~the~transition probability~$P$. 

\begin{definition}[irreducibility]
The chain is $\psi$-\textit{irreducible} if there exists a $\sigma$-finite
measure $\psi$ such that, for all set $B\in\mathcal{E}$, when $\psi(B)>0$, for any $x\in E$ there exists $n>0$ such that $P^n(x,B)>0$. With words, no matter the starting
point is, the chain visits $B$ with strictly positive probability. 
\end{definition}

\begin{definition}[aperiodicity]
 Assuming $\psi$-irreducibility, there exists $d^{\prime}\in\mathbb{N}^{\ast
}$ and disjoints sets $D_{1},....,$ $D_{d^{\prime}}$ (set $D_{d^{\prime}+1}=D_{1}%
$) positively weighted by $\psi$ such that $\psi(E\backslash\cup_{1\leqslant i\leqslant
d^{\prime}}D_{i})=0$ and $\forall x\in D_{i},$ $P(x,D_{i+1})=1.$ The
\textit{period} of the chain is the g.c.d. $d$ of such integers, it is said to be
\textit{aperiodic} if $d=1$. 
\end{definition}

\begin{definition}[Harris recurrence] Given a set $B \in E$ and $\tau_{B}$ the time the chain first enters $B$, a $\psi$-irreducible Markov chain is said to be \textit{positive Harris recurrent} if for all
$B \in E$ with $\psi(B) > 0$,  we have $\mathbb{E}_{x}\tau_{B} < \infty $ for all $x \in B$.
\end{definition}

Recall that a chain is positive Harris recurrent and aperiodic if and only if it is ergodic \citep[Proposition 6.3]{nummelin:1984}, i.e., there exists a probability measure $\pi$, called the stationary distribution, such that $\lim_{n\to +\infty} \|P^n(x,\cdot) - \pi \|_{\text{tv}}  = 0$. The Nummelin splitting technique (presented in the forthcoming section) depends heavily on the notion of small set. Such sets exist for positive Harris recurrent chain \citep{JainJam}.

\begin{definition}[small sets]
A~set~$S\in\mathcal{E}$ is said to be $\Psi$-small if there exists
$\delta>0,$ a~positive probability measure~$\Psi$ supported by~$S$ and
an~integer $m\in \mathbb N^{\ast}$ such that
\begin{equation}
\forall x\in S,\;B\in\mathcal{E}\;\;P^{m}(x,B)\geq\delta\;\Psi
(B).\label{eq:minorization}%
\end{equation}
\end{definition}

In the whole paper, we work under the following generic hypothesis in which the chain is supposed to be Harris recurrent. Let $\mathbb{P}_{x}$ (resp. $\mathbb{P}_{\nu})$ denote
the~probability measure such that $X_{0}=x$ and $X_{0}\in E$ (resp. $X_{0}%
\sim\nu$), and~$\mathbb{E}_{x}\left(  \cdot\right)  $ is~the~$\mathbb{P}_{x}%
$-expectation (resp. $\mathbb{E}_{\nu}\left(  \cdot\right)  $ the
$\mathbb{P}_{\nu}$-expectation).

\begin{enumerate}[(\text{H})]
\item \label{generic_assumption} The chain $(X_{n})_{n\in\mathbb{N}}$ is a positive Harris recurrent aperiodic Markov chain
with countable state space $(E,\mathcal{E})$, transition kernel $P(x,dy)$ and initial measure $\nu$. Let $S$ be $\Psi$-small with $m=1$ and suppose that the hitting time $\tau_{S}$ satisfies
\begin{equation}
\sup_{x\in S}\mathbb{E}_{x}[\tau_{S}]<\infty,\qquad \text{and} \qquad \mathbb{E}_{\nu}[\tau_{S}]<\infty.\label{tauint}
\end{equation}
\end{enumerate}

This is only for clarity reasons that we assume that $m=1$. As explained in Remark \ref{remark:1_dependent} below, the study of sums over Harris chain, i.e., when $m>1$, can easily be derived from the case $m=1$. 


\subsection{The Nummelin splitting technique}

The Nummelin splitting technique \citep{nummelin:1978,athreya+n:1978} allows to retrieve all regeneration properties for general Harris Markov chains. It consists in extending the probabilistic structure of the
chain in order to construct an artificial atom \citep{Nummelin84}. Start by recalling the definition of regenerative chains.
 
\begin{definition}
We say that a $\psi$-irreducible $X$, aperiodic chain is \textit{regenerative} or \textit{atomic} if there
exists a measurable set~$A$ called an atom, such that $\psi(A)>0$ and for all~$(x,y)\in A^{2}$
we have~$P(x,\cdot)=P(y,\cdot)$.  Roughly speaking, an atom is a
set on which the transition probabilities are the same.\ If the chain visit a
finite number of states then any state or any subset of the states is actually
an atom.
\end{definition}


Assume that the chain $X$ satisfies the generic hypothesis (H). Then the sample space is expanded in order to~define a sequence $(Y_{n}%
)_{n\in\mathbb{N}}$ of independent Bernoulli random variables with parameter $\delta$. The construction relies on the mixture
representation of $P$ on $S,$ namely $P(x,A)=\delta\Psi(A)+(1-\delta
) ({P(x,A)-\delta\Psi(A))}/{(1-\delta)}$, with two components, one of which not
depending on the starting point (implying regeneration when this component is
picked up in the mixture)$.$ The regeneration structure can be retrieved by
the following randomization of the transition probability $P$ each time the
chain $X$ visits the set $S$ : 

\begin{itemize}
\item If $X_{n}\in S$ and $Y_{n}=1$ (which happens with probability $\delta\in\left]
0,1\right[  $), then $X_{n+1}$ is distributed according to the probability
measure $\Psi$,

\item If $X_{n}\in S$ and $Y_{n}=0$ (that happens with probability $1-\delta$), then $X_{n+1} $
is distributed according to the probability measure $(1-\delta)^{-1}(P
(X_{n},\cdot)-\delta\Psi(\cdot)).$
\end{itemize}

The bivariate Markov chain $Z=(X_{n},Y_{n})_{n\in\mathbb{N}}$ is called the \textit{split
chain}. It takes its values in $E\times\left\{  0,1\right\}  $ and is atomic with atom given by  $A=S\times\left\{  1\right\}  $.  Define the sequence of \textit{regeneration times} $(\tau_{A}(j))_{j\geq1},$ i.e.
\[
\tau_{A}=\tau_{A}(1)=\inf\{n\geq1\,  : \,  Z_{n}\in A\}
\]
and, for $ j\geq2$,
\[
\tau_{A}(j)=\inf\{n>\tau_{A}(j-1)\, : \,Z_{n}\in A\}.
\]
It is well known that the bivariate chain~$Z$ inherits all the stability and communication properties of~the chain~$X$, as aperiodicity and $\psi
$-irreducibility. For instance, the
regeneration times has a finite expectation (by recurrence property), more precisely, it holds that \citep[Lemma 9]{azais:2016}
\begin{align*}
\sup_{x\in A}\mathbb{E}_{x}[\tau_{A}]<\infty\qquad \text{and} \qquad \mathbb{E}_{\nu}[\tau_{A}]<\infty.
\end{align*}
It is known from
regeneration theory \citep{meyn+t:2009} that given sequence
$(\tau_{A}(j))_{j\geq1},$ we can cut our data into block segments or cycles defined by 
\[
{B}_{j}=(X_{1+\tau_{A}(j)},\cdots,X_{\tau_{A}(j+1)}),\;j\geq1
\]
according to the consecutive visits of the chain to the regeneration set~$A$.
The strong Markov property the sequences implies that  $(\tau
_{A}(j))_{j\geq1}$ and $({B}_{j})_{j\geq1}$ are i.i.d. \citep[Lemma 3.1]{bednorz+l+l:2008}. 
Denote by $\mathbb P_A$ the probability measure such that $X_0\in A$. The stationary
distribution is given by the Pitman's occupation measure:
\[
\pi(B)=\frac{1}{\mathbb{E}_{A}(\tau_{A})}\left(  \sum_{i=1}^{\tau_{A}%
}\mathbb{I}\{X_{i}\in B\}\right)  ,\;\forall B\in\mathcal{E},
\]
where $\mathbb{I}\{{B}\}$ is the~ indicator function
of~the~event $B$. Let $f:(E,\mathcal{F})\rightarrow\mathbb{R}$ be a general measurable function.
In the following we consider partial sums over regenerative cycles
$f'(B_{j})=\sum_{j=\tau_{A}(i)+1}^{\tau_{A}(i+1)}f(X_{i})$. We denote by
$l_{n}=\sum_{i=1}^{n}\mathbb{I}\{X_{i}\in A\}$ the total number of renewals,
thus we observe $l_{n}+1$ blocks.  Notice that the block length are also i.i.d. with mean $\mathbb{E}_{A}(\tau_{A})$.

\begin{remark}[random number of blocks]
The number of blocks $l_{n\ }$is
random and correlated to the blocks itself. This causes a major difficulty
when deriving second order asymptotic results as well as non-asymptotic results for regenerative
Markov chains.
\end{remark}

\begin{remark}[small set or atom]
The Nummelin splitting technique is useless in the case where the initial chain $X$ is already atomic, in which case the atom is simply $A = S$. For clarity, we choose to focus on the general framework of Harris chains.
\end{remark}

\section{Block Rademacher complexity}\label{sec:Rad_markov}

\subsection{The independent case}

 Let $\xi = (\xi_{i})_{i\in\mathbb{N}}$ be an i.i.d. sequence of random variables defined on $(\Omega,\mathcal{F},\mathbb{P})$ valued in $(E,\mathcal E)$ with common
distribution $P$ on $(E,\mathcal E)$. Let $\mathcal{F}$ be a countable class of real-valued measurable
functions defined on $E$. The Rademacher complexity associated to
$\mathcal{F}$ is given by
\[
R_{n,\xi}(\mathcal{F})=\mathbb{E}\sup_{f\in\mathcal{F}}\left\vert \sum
_{i=1}^{n}\epsilon_{i}f(\xi_{i})\right\vert ,
\]
where the ($\epsilon_{i})_{i\in \mathbb N}$ are i.i.d. Rademacher random
variables, i.e., taking values $+1$ and $-1$, with probability $1/2$,
independent from $\xi$. 

The notion of VC class is powerful because it covers many interesting classes of functions and ensures suitable properties on the Rademacher complexity. The function $F$ is an envelope for the class $\mathcal{F}$ if $|f(x)|\leq
F(x)$ for all $x\in E$ and $f\in\mathcal{F}$. For a metric space
$(\mathcal{F},d)$, the covering number $\mathcal{N}(\epsilon,\mathcal{F},d)$ is
the minimal number of balls of size $\epsilon$ needed to cover $\mathcal{F}
$. The metric that we use here is the
$L_{2}(Q)$-norm denoted by $\Vert.\Vert_{L_{2}(Q)} \ $ and given by $\ \Vert f\Vert_{L_{2}(Q)}=\{\int{f^{2}dQ}\}^{1/2}$. 

\begin{definition}
A class $\mathcal{F}$ of measurable functions $E\to \mathbb R$ is said to be of VC-type (or
Vapnik-Chervonenkis type) for an envelope $F$ and admissible characteristic $(C,v)$
(positive constants) such that $C\geq(3\sqrt
{e})^{v}$ and $v\geq1$, if for all probability measure $Q$ on $(E,\mathcal{E})$
with $0<\Vert F\Vert_{L_{2}(Q)}<\infty$ and every $0<\epsilon<1$,
\[
\mathcal{N}\left(  \epsilon\Vert F\Vert_{L_{2}(Q)},\,\mathcal{F},\,\Vert
.\Vert_{L_{2}(Q)}\right)  \leq{C}{\epsilon^{-v}}.
\]
We also assume that the class is countable to avoid measurability issues (but the non-countable case may be handled similarly by using outer probability and additional measurability assumptions, see \cite{van+w:2007}).
\end{definition}

The next theorem is taken from \cite{gigui2001}, Proposition 2.2, and has
been successfully applied to kernel density estimators in \cite{gine+g:02}. Similar approaches are provided for instance in \cite{einmas}, Proposition 1.

\begin{theorem}
[\cite{gigui2001}]\label{th:indep_case} Let $\mathcal{F}$ be a measurable
uniformly bounded VC class of functions defined on $E$ with envelop $F$ and
characteristic $(C,v)$. Let $U>0$ such that
$|f(x)|\leq U$ for all $x\in E$ and
$f\in\mathcal{F}$. Let $\sigma^{2}$ be such that $\mathbb{E}[f(\xi)^{2}%
]\leq\sigma^{2}$ for all $f\in\mathcal{F}$. Then, whenever $0<\sigma\leq U$,
it holds
\[
R_{n,\xi}(\mathcal{F})\leq M\left[  vU\log\frac{CU}{\sigma}+\sqrt
{vn\sigma^{2}\log\frac{CU}{\sigma}}\right]  ,
\]
where $M$ is a universal constant.
\end{theorem}



\subsection{The Harris case}

To extend the previous approach to any Harris chain $X$, we decompose the chain according to the elements $X_{i}$ that belong to
complete blocks $B_{1},\ldots,B_{l_{n}-1}$ and the elements $X_i$ in $B_{0}$ and
$B_{l_{n}}$. Assuming that $l_n>1$,
\begin{align}\label{decomp:complete_incomplet_block}
&\sup_{f\in\mathcal{F}}\left\vert \sum_{i=1}^{n}%
(f(X_{i})-\mathbb{E}_{\pi}[f])\right\vert\leq \\
&\sup_{f\in\mathcal{F}}\left\vert \sum_{i=\tau_{A}
(1)+1}^{\tau_{A}(l_{n})}(f(X_{i})-\mathbb{E}_{\pi}[f])\right\vert +\sup
_{f\in\mathcal{F}}\left\vert \sum_{i=1}^{\tau_{A}}(f(X_{i})-\mathbb{E}%
_{\pi}[f])\right\vert +\sup_{f\in\mathcal{F}}\left\vert \sum_{i=\tau_{A}
(l_{n})+1}^{n}(f(X_{i})-\mathbb{E}_{\pi}[f])\right\vert ,\nonumber
\end{align}
with the convention that empty sums are $0$.
The terms corresponding to the first and last blocks say $B_{0}$ and
$B_{l_{n}}$, will be treated separately. Because $\tau_A(l_{n})-\tau_A(1)=\sum_{k=1}^{l_{n}}%
\ell(B_{k})$, where $\ell(B_k)$ denote the size of block $k$, it holds that
\[
\left\vert \sum_{i=\tau_{A}(1)+1}^{\tau_{A}(l_{n})}(f(X_{i})-\mathbb{E}_{\pi
}[f])\right\vert =\left\vert \sum_{k=1}^{l_{n}}(f^{\prime}(B_{k})-\ell
(B_{k})\mathbb{E}_{\pi}[f])\right\vert .
\]
where $f'(B_k) = \sum_{i=\tau_A(k)+1}^{\tau_A(k)} f(X_i)$. Hence this term is a (random) summation over complete blocks. Recall that $l_{n}\leq
n$ and that, under (\ref{tauint}), ${l_{n}}/ {n}\rightarrow1/\mathbb{E}_{A}%
\tau_{A}$. Thus, aiming to reproduce the Rademacher approach in the i.i.d. setting, we introduce the following block
Rademacher complexity of the class $\mathcal{F}$,
\[
R_{n,B}(\mathcal{F})=\mathbb E _A\sup_{f\in\mathcal{F}}\left\vert \sum_{k=1}%
^{n}\epsilon_{k}f^{\prime}(B_{k})\right\vert ,
\]
where $(\epsilon_{k})_{k\in\mathbb{N}}$ are Rademacher random variables independent from the blocks $(B_{k}) _{k\in\mathbb{N}}$.

\subsection{Block VC classes}

Even if the blocks $(B_{k})$ form an independent
sequence, we cannot apply directly concentration results for empirical
processes over bounded classes, e.g., Theorem \ref{th:indep_case}, because the
class of functions $\mathcal{F}^{\prime }$ formed by $f^{\prime }$ is not
bounded. To solve this problem we will show that it is possible by an adequate probability transformation to bound the covering number of the functions $\mathcal{F}^{\prime }$ by the one of the original class of functions $\mathcal{F}$ for an adequate metric. In particular, we show that the class of functions
$\mathcal{F}^{\prime}$ has a similar size, in terms of covering number, as the
class $\mathcal{F}$. This in turn will help to extend existing concentration
inequalities on $\mathcal{F}$ to concentration inequalities on $\mathcal{F}%
^{\prime}$. 

For this define $E^{\prime}= \cup_{k=1} ^{\infty}E^{k}$ and let the \textit{occupation
measure} $M$ be given by
\begin{align*}
M(B, dy ) = \sum_{x\in B} \delta_{x}(y), \qquad\text{for every } B\in
E^{\prime}.
\end{align*}
Introduce the following notations : for any function $f: E\to \mathbb R$, let $f':E^{\prime} \to \mathbb R$ be given by 
\begin{align*}
f^{\prime}(B) = \int f(y) \, M(B,\mathrm{d}
y) = \sum_{x\in B} f(x),
\end{align*}
and for any class $\mathcal{F}$ of real-valued functions defined on $E$, denote
by 
\begin{align}
\label{def:classFprime} \mathcal F '  = \{ f^{\prime}\,:\, f\in \mathcal F\} .
\end{align}
The function that gives the size of the blocks $\ell$ is $\ell: E^{\prime}%
\to\mathbb{N}^{*} $, defined by,
\begin{align*}
\ell(B) = \int\, M(B,\mathrm{d} y),\qquad\text{ for every } B\in E^{\prime}.
\end{align*}
Let $\mathcal{E}^{\prime}$ denote the smallest $\sigma$-algebra formed by the
elements of the $\sigma$-algebras $\mathcal{E}^{k}$, $k\geq1$, where
$\mathcal{E}^{k}$ stands for the classical product $\sigma$-algebra. Let
$Q^{\prime}$ denote a probability measure on $(E^{\prime}, \mathcal{E}%
^{\prime})$. If $B(\omega)$ is a random variable with distribution $Q^{\prime
}$, then $M(B(\omega),\mathrm{d} y)$ is a random measure, i.e., $M(B(\omega),
\mathrm{d} y )$ is a (counting) measure on $(E,\mathcal{E})$, almost surely,
and for every $A\in\mathcal{E}$, $M(B(\omega),A) = \int_{A} \, M(B(\omega
),\mathrm{d} y) $ is a measurable random variable (valued in $\mathbb{N}$).
Henceforth $\ell(B(\omega))\times\int f(y) \, M(B(\omega),\mathrm{d} y)$ is a
random variable and, provided that $Q^{\prime}(\ell^{2} ) <\infty$, the map
$Q$, defined by
\begin{align}
\label{def:proba_measure}Q(A) = \mathbb{E}_{Q^{\prime}} \left(  \ell(B)
\times\int_{A} \, M(B,\mathrm{d} y)\right)  / Q^{\prime}( \ell^{2})
,\qquad\text{for every } A\in\mathcal{E},
\end{align}
is a probability measure on $(E,\mathcal{E})$.

\begin{lemma}
	\label{lemma:covering+VC} Let $Q^{\prime}$ be a probability measure on $(E^{\prime}, \mathcal{E}')$ such that $0< \|\ell\|_{L_{2}(Q^{\prime})}<\infty$ and $\mathcal{F}$ be a class of measurable real-valued functions defined on $E$. 
	Then we have, for every $0<\epsilon<\infty$,
	\begin{align*}
	\mathcal{N }(\epsilon\|\ell\| _{L_{2}(Q^{\prime})},\, \mathcal{F}^{\prime},\,
	L_{2}(Q^{\prime}) ) \leq\mathcal{N }(\epsilon,\, \mathcal{F},\, L_{2}(Q) ),
	\end{align*}
	where $\mathcal{F}^{\prime}$ and $Q$ are given in (\ref{def:classFprime}) and
	(\ref{def:proba_measure}), respectively. Moreover if $\mathcal{F}$ is VC with constant envelope $U$ and characteristic $(C,v)$, then $\mathcal{F}^{\prime}$ is VC
	with envelope $U \ell $ and characteristic $(C,v)$. 
	
\end{lemma}

\begin{proof}
	The proof is inspired from the proof of Lemma 4.2 presented in \cite{levental:1988}. Let $f'$ be such that (\ref{def:classFprime}) holds with $f$ a $\mathcal E$-measurable function. Then, using Jensen's inequality,
	\begin{align*}
	Q'(f'^2) &= \mathbb E_{Q'}\left( \left(\int  f (y) \, M(B,d y)\right)^2 \right) \\
	&\leq \mathbb E_{Q'} \left( \ell (B)  \left(\int  f (y)^2 \, M(B,d y)\right) \right) \\
	& = Q(f^2) Q'(\ell ^2).
	\end{align*}
	Applying this to the function
	\begin{align*}
	f'(B) - f_k'(B) = \int  (f (y)-f_k(y)) \, M(B,d y),
	\end{align*}
	when each $f_k$ is the center of an $\epsilon$-cover of the space $\mathcal F$ and $\|f-f_k\|_{L_2(Q)}\leq \epsilon$ gives the first assertion of the lemma.
	To obtain the second assertion, note that $F' = U\ell$ is an envelope for $\mathcal F'$. In addition, we have that
	\begin{align*}
	\|F'\|_{L_2(Q')} = U \|\ell\|_{L_2(Q')} .
	\end{align*}
	From this we derive that, for every $0<\epsilon<1$,
	\begin{align*}
	\mathcal N (\epsilon \|F'\| _{L_2(Q')},\, \mathcal F',\, L_2(Q') ) =  \mathcal N (\epsilon U  \|\ell\|_{L_2(Q')},\, \mathcal F',\, L_2(Q') ).
	\end{align*}
	Then using the first assertion of the lemma, we obtain for every $0<\epsilon<1$,
	\begin{align*}
	\mathcal N (\epsilon \|F'\| _{L_2(Q')},\, \mathcal F',\, L_2(Q') )\leq \mathcal N (\epsilon U,\, \mathcal F,\, L_2(Q) ),
	\end{align*}
	which implies the second assertion whenever the class $\mathcal F$ is VC for the envelope $F$.
\end{proof}

Now that we know that any bounded VC class $\mathcal F$ can be extended to a VC class $\mathcal F'$ unbounded defined over the blocks, we consider the bounded case $\mathcal{F}^{\prime} \mathrm 1_{\{ \ell \leq L\}} = \{f' \mathrm 1_{\{ \ell \leq L\}} \,:\, f\in \mathcal F\}$ which, unsurprisingly, is shown to remain VC. 

\begin{lemma}
	\label{lemma:covering+VC_bounded} Let $Q^{\prime}$ be a probability measure on $(E^{\prime}, \mathcal{E}')$ and $\mathcal{F}$ be a class of measurable real-valued functions defined on $E$. 
	Then we have, for every $0<\epsilon<\infty$,
	\begin{align*}
	\mathcal{N }(\epsilon L ,\, \mathcal{F}^{\prime} \mathrm 1_{\{ \ell \leq L\}},\,
	L_{2}(Q^{\prime}) ) \leq\mathcal{N }(\epsilon,\, \mathcal{F},\, L_{2}(\tilde Q) ),
	\end{align*}
	where $\tilde Q =  \mathbb{E}_{Q^{\prime}} \left(  \ell(B)\mathrm 1_{\{ \ell(B) \leq L\}}
\times\int_{A} \, M(B,\mathrm{d} y)\right)  / Q^{\prime}( \ell^{2}\mathrm 1_{\{ \ell \leq L\}})$. Moreover if $\mathcal{F}$ is VC with constant envelope $U$ and characteristic $(C,v)$, then $\mathcal{F}^{\prime} \mathrm 1_{\{ \ell \leq L\}}$ is VC
	with envelope $ L U$ and characteristic $(C,v)$. 
	
\end{lemma}


\begin{proof}
The proof follows the same lines as the proof of Lemma \ref{lemma:covering+VC}, replacing $\ell $ by $\ell \mathrm 1_{\{\ell \leq L\}}$.

\end{proof}

\section{Main result}\label{sec:main_results}

 We shall distinguish between the two following assumptions on the regeneration time $\tau_A$. We say that $\tau_A$ has polynomial moments, whenever
\begin{itemize}\itemsep10pt
\item[(PM)]\label{pm} there exists $p>1$ such that $\mathbb E_A[\tau_A^{p}]<\infty$,
\end{itemize}
and has some exponential moments (EM), as soon as
\begin{itemize}\itemsep10pt
\item[(EM)]\label{em2}  there exists $\lambda>0$ such that
$\mathbb{E}_{A}[\exp(\tau_A\lambda)]<\infty$.
\end{itemize}
 The following result
extends concentration inequalities for empirical processes over independent
random variables \citep{gigui2001,gine+g:02,einmas}, e.g., Theorem
\ref{th:indep_case}, to Markov chains.

\begin{theorem}[block Rademacher complexity]\label{th:rademacher_complexity}
Assume that the chain $X$ satisfies the generic hypothesis (H).
Let $\mathcal{F}$ be VC with constant envelope $U$ and characteristic $(C,v)$. Let
$\sigma^{\prime 2}$ be such that 
\begin{align*}
\mathbb E_A\left[\left(\sum_{i=1}^{\tau_A} f(X_i)\right)  ^{2}\right ]\leq
\sigma^{\prime 2}, \qquad  \text{for all }f\in\mathcal{F}.
\end{align*}
 For some universal constant $M>0$, and any $L$ such that $0<\sigma^{\prime}\leq LU$,

\begin{enumerate}[(i)]
\item \label{th:first_point} if (PM) holds, then
\[
R_{n,B}(\mathcal{F})\leq M\left[  vLU\log\frac{CL U}%
{\sigma^{\prime}}+\sqrt{vn\sigma^{\prime2}\log\frac{CLU}%
{\sigma^{\prime}}}\right]  +\frac{n\mathbb E_A[\tau_A^{p}]}{L^{p-1}},
\]

\item \label{th:second_point} if (EM) holds, then
\[
R_{n,B}(\mathcal{F})\leq M\left[  v LU\log\frac{CL%
U}{\sigma^{\prime}}+\sqrt{v n\sigma^{{\prime}2}\log\frac{C L%
U}{\sigma^{\prime}}}\right]  +nU\exp(-L\lambda/2)C_{\lambda},
\]
where $C_{\lambda}={2}\mathbb E_A[\exp(\tau_A \lambda)]/\lambda$.
\end{enumerate}
\end{theorem}

\begin{proof}
First we show that
\begin{equation}\label{ineq:intermed}
R_{n,B}(\mathcal{F})\leq M\left[ vLU\log \frac{CLU }{\sigma
^{{\prime }2}}+\sqrt{vn\sigma ^{{\prime }2}\log \frac{CLU }{\sigma
^{{\prime }2}}}\right] +nU \mathbb E_A[\tau_A
\mathrm{1}_{\{ \tau_A>L\}}],
\end{equation}
for some universal constant $M>0$. Then we consider the two cases (\ref{th:first_point}) and (\ref{th:second_point}) to bound $\mathbb E_A[\tau_A
\mathrm{1}_{\{ \tau_A>L\}}]$
accordingly.\newline
 Use the decomposition
\begin{equation}\label{decomp:bounded_unbounded_terms}
\sum_{k=1}^{n}\epsilon _{k}f^{\prime }(B_{k})=\sum_{k=1}^{n}\epsilon _{k}%
\underline{f}_{L}^{\prime }(B_{k})+\sum_{k=1}^{n}\epsilon _{k}\overline{f}%
_{L}^{\prime }(B_{k}),
\end{equation}
where, for any $B\in E^{\prime }$,
\begin{align*}
\underline{f}_{L}^{\prime }(B)=f^{\prime }(B)\mathrm{1}_{\{\ell (B)\leq L\}},\\
\overline{f}_{L}^{\prime }(B)=f^{\prime }(B)\mathrm{1}_{\{\ell (B)>L\}}.
\end{align*}
The first term in (\ref{decomp:bounded_unbounded_terms}) represents a
classical Rademacher complexity as it is a centered empirical process
evaluated over the bounded class $\mathcal{F}' \mathrm{1}_{\{\ell (B)\leq L\}}$. It follows from Lemma \ref{lemma:covering+VC_bounded} that the product class $\mathcal{F}' \mathrm{1}_{\{\ell (B)\leq L\}} $ is VC with constant envelop $L U$. As by assumption, $0<\sigma ^{\prime }\leq LU$,
we deduce from applying Theorem \ref{th:indep_case} (with $LU$
in place of $U$), that
\begin{equation*}
\mathbb{E}_A\sup_{f\in \mathcal{F}}\left\vert \sum_{k=1}^{n}\epsilon _{k}\,%
\underline{f}_{L}^{\prime }(B_{k})\right\vert \leq  M\left[ vLU\log \frac{%
CL U}{\sigma' }+\sqrt{vn\sigma ^{\prime2}\log \frac{CLU}{\sigma' }}\right]
.
\end{equation*}%
For the second term in (\ref{decomp:bounded_unbounded_terms}), we find
\begin{equation*}
\mathbb{E}_A\sup_{f\in \mathcal{F}}\left\vert \sum_{k=1}^{n}\epsilon _{k}%
\overline{f}_{n}^{\prime }(B_{k})\right\vert \leq nU\mathbb E_A[\ell (B_{1})%
\mathrm{1}_{\{\ell (B_{1})>L\}}] =nU \mathbb E_A[\tau_A
\mathrm{1}_{\{ \tau_A>L\}}]
\end{equation*}%
Hence (\ref{ineq:intermed}) is established. To obtain point (\ref%
{th:first_point}) simply use Markov's inequality. To obtain (\ref%
{th:second_point}), note that
\begin{equation*}
\mathbb E_A[\tau_A
\mathrm{1}_{\{ \tau_A>L\}}] \leq \exp (-L\lambda
/2)\mathbb E_A[\tau_A \exp (\tau_A\lambda /2)]\leq \frac{2}{\lambda }\exp
(-L\lambda /2)\mathbb E_A[\exp (\tau_A\lambda )].
\end{equation*}%
The last inequality follows from
$\exp (t\lambda /2){\lambda } /2 \leq \exp (t\lambda )$ which is implied by $x\leq \exp(xt)$ whenever $t\geq 1$.
\end{proof}

\begin{remark}[geometric ergodicity and condition (EM)]\label{remark:geo_ergo}
Condition (EM) is equivalent to each of the following assertions : (i) the geometric ergodicity of the chain $X$, (ii) the (uniform) Doeblin condition, as well as (iii) the Foster-Lyapunov drift condition (see Theorem 16.0.2 in \cite{meyn+t:2009} for the details). Under this assumption, most classical convergence results (for instance, the law of the iterated logarithm or the central limit theorem) are valid \cite[Chapter 17]{meyn+t:2009}.
\end{remark}

\begin{remark}[mixing and (PM)]
We point out that the relationship between (PM) and the rate of decay of mixing coefficients has been
investigated in Bolthausen (1982): this condition is typically fulfilled as
soon as the strong mixing coefficients sequence decreases as an arithmetic
rate $n^{-s}$, for some $s>p-1$.
\end{remark}

\begin{remark}[choice of the atom]
  Finding $A$ in practice can be done by plotting an estimator of the transition
density and finding a zone were the density is lower bounded (in practice,
$\Psi$ may be simply chosen to be the$\ $uniform distribution over the small set).
\end{remark}

The two following results  show that the block Rademacher complexity, previously introduced, is useful to control the
expected values as well as the excess probability of suprema over classes of functions. 

\begin{theorem}
[expectation bound]\label{th:expectation_bound} Assume that the chain $X$ satisfies the generic hypothesis (H).
Let $\mathcal{F}$ be a countable class
of measurable functions bounded by $U$. It holds that
\[
\mathbb{E}_{\nu}\left[  \sup_{f\in\mathcal{F}}\left\vert \sum_{i=1}%
^{n}(f(X_{i})-\mathbb{E}_{\pi}[f])\right\vert \right]  \leq  4R_{n,B}(\mathcal{F}) + 4 \sup_{f\in \mathcal F} | \mathbb E _\pi [f]| \sqrt { n \mathbb E_A[\tau_A^2] } +2U(\mathbb{E}_{\nu}[\tau_{A}]+\mathbb E_A
[\tau_{A} ]),
\]
where $\nu$ stands for the initial measure.
\end{theorem}

\begin{proof}

We rely on the block decomposition (\ref{decomp:complete_incomplet_block}%
). First, we apply Lemma 1.2.6 in \cite{delapena+g:1999} to treat the term
formed by complete blocks. Denote by $\mathcal{F}_{c}^{\prime}$ the
class formed by $\{f^{\prime}-\ell \, \mathbb{E}_{\pi
}[f]\}$. We obtain
\begin{align*}
\mathbb{E}_{\nu}\left[  \sup_{f\in\mathcal{F}}\left\vert \sum_{i=\tau
_{A}(1)+1}^{\tau_{A}(l_{n})}(f(X_{i})-\mathbb{E}_{\pi}[f])\right\vert \right]
&  \leq  \mathbb{E}_A\left[  \max_{1\leq l\leq n}\sup_{f\in\mathcal{F}}\left\vert
\sum_{k=1}^{l_{n}}\{f^{\prime}(B_{k})-\ell(B_{k})\mathbb{E}_{\pi}[f]\}\right\vert
\right]   \\
&  \leq4\mathbb{E}_{\epsilon,A} \left[  \sup_{f\in\mathcal{F}}\left\vert
\sum_{k=1}^{n}\epsilon_{k}\{f^{\prime}(B_{k})-\ell(B_{k})\mathbb{E}_{\pi
}[f]\}\right\vert \right] \\
&  =4R_{n,B}(\mathcal{F}_{c}^{\prime}).
\end{align*}
From the triangular inequality and because 
\begin{align*}
\mathbb{E}_{\epsilon,A} \left[  \sup_{f\in\mathcal{F}}\left\vert
\sum_{k=1}^{n}\epsilon_{k} \ell(B_{k})\mathbb{E}_{\pi
}[f] \right\vert \right]\leq \sup_{f\in \mathcal F} | \mathbb E _\pi [f]|  \sqrt { n \mathbb E_A[\tau_A^2] },
\end{align*}
we obtain that $R_{n,B}(\mathcal{F}_{c}^{\prime})\leq R_{n,B}(\mathcal{F}^{\prime}) + \sup_{f\in \mathcal F} | \mathbb E _\pi [f]|  \sqrt { n \mathbb E_A[\tau_A^2] }$.
The terms corresponding to incomplete blocks are treated as follows. We have
\begin{align*}
&  \mathbb{E}_{\nu}\sup_{f\in\mathcal{F}}\left\vert \sum_{i=1}^{\tau_{A}%
(1)}(f(X_{i})-\mathbb{E}_{\pi}[f])\right\vert \leq2U\mathbb{E}_{\nu}[\tau
_{A}],\\
&  \mathbb{E}_{\nu}\sup_{f\in\mathcal{F}}\left\vert \sum_{i=\tau_{A}(l_{n}%
)}^{n}(f(X_{i})-\mathbb{E}_{\pi}[f])\right\vert \leq2U\mathbb E_A[\tau_{A}].
\end{align*}

\end{proof}


Using Theorem \ref{th:expectation_bound}, we now rephrase the result of \cite{adamczak:2008} to obtain a concentration bound for the empirical process involving the Rademacher complexity $R_{n,B}(\mathcal{F}^{\prime})$ defined previously.

\begin{theorem}[concentration bound, \cite{adamczak:2008}]\label{th:proba_bound}
Assume that the chain $X$ satisfies the generic hypothesis (H), (EM) and there exists $\lambda >0$ such that $\mathbb E_\nu[\exp(\lambda\tau_A)] <\infty$. Let $\mathcal{F}$ be a countable class of measurable functions bounded by $U$. Let $R_n$ be such that
\begin{align*}
&R_{n}\geq 4R_{n,B}(\mathcal{F}) + 4 \sup_{f\in \mathcal F} | \mathbb E _\pi [f]| \sqrt { n \mathbb E_A[\tau_A^2] } +2U(\mathbb{E}_{\nu}[\tau_{A}]+\mathbb E_A
[\tau_{A} ]),\\
& \sigma^{\prime 2}\geq \sup_{f\in\mathcal{F}}\mathbb E _{A}\left[\left(  \sum_{i=1}^{\tau_{A}}
f(X_{i})\right)^2\right].
\end{align*}
Then, for some
universal constant $K>0$, and for $\tau>0$ depending on the tails of the regeneration time, we have, for all $t\geq1+KR_{n}$, 
\[
\mathbb{P}_{\nu}\left(  \sup_{f\in\mathcal{F}}\left\vert \sum_{i=1}^{n}%
(f(X_{i})-\mathbb{E}_{\pi}(f) ) \right\vert \geq t\right)  \leq K\exp\left[
-\frac{\mathbb{E}_{A}[\tau_{A}]}{K}\min\left(  \frac{(t-KR_{n})^{2}}{n \sigma^{\prime 2}},\frac
{(t-KR_{n})}{\tau^{3}Ulog\;n}\right)  \right],
\]
yielding alternatively, that for any $n/\log(n) \geq {\tau^{3}U} /   \sigma^{\prime 2}$ with probability $1-\delta$ we have,
\[
\sup_{f\in\mathcal{F}}\left\vert \sum_{i=1}^{n}(f(X_{i})-\mathbb{E}_{\pi
}(f) ) \right\vert \leq KR_{n}+\max\left(  \sqrt{n}  \sigma^\prime  \sqrt{K\log\left(\frac
{K}{\delta}\right)},\log\left(\frac{K}{\delta}\right)\frac{\tau^{3}U\log(n)}{\mathbb{E}_{A}[\tau_{A}]}\right).
\]

\end{theorem}

\begin{remark}[on Theorem \ref{th:proba_bound}]
An explicit value for the constant $K$\ is difficult to obtain from
the results of Adamczak(2008)\ but would be of great interest in practical
applications. Notice that for $n$ large the second member of the inequality
reduces to the bound $KR_{n}+\sqrt{n}\sigma' \sqrt{K\log( {K}/{\delta})}
$, which gives the same rate as in the i.i.d. case. 
\end{remark}

\begin{remark}[$m$ different from $1$]\label{remark:1_dependent}
We have reduced our analysis to the case $m=1$, however it is very easy to see now how the general case $m>1$ can be handled up to a modified constant in the bound. Recall that when $m>1$ then the blocks $f(B_{i})$  are 1-dependent (see for instance \cite{chen:1999} Corollary 2.3). It follows that we can split the sum as follows
$$
\sum_{k=0}^{l_{n}}f(B_{i})=\sum_{k=0, \, k\, \text{even}}^{l_{n}} f(B_{k})+\sum_{k=0, \, k\, \text{odd}}^{l_{n}}%
f(B_{k})$$
Then notice that, because of the $1$-dependence property,  in each sums the
blocks are independent and we now have two sums of at most $n/2$
independent blocks that can be treated separately as we did before.

\end{remark}

\section{Applications}\label{sec:applications}

\subsection{Kernel density estimator}\label{sec:applications_1}

Given $n\geq 1$ observations of a Markov chains $X\subset \mathbb R^d$, the kernel density estimator of the stationary measure $\pi$ is
given by
\[
\hat{\pi}_{n}(x)=n^{-1}\sum_{i=1}^{n}K((x-X_{i})/h_{n})/h_{n}^{d},
\]
where $K:\mathbb{R}^{d}\rightarrow\mathbb{R}$, called the kernel, is such that
$\int K(x)dx=1\ $and $(h_{n})_{n\geq 1}$ is a positive sequence of bandwidths. 

The analysis of the asymptotic behavior of $\hat{\pi}_{n}-\pi$ is
traditionally executed by studying two terms. The bias term, $\mathbb{E}\hat{\pi}_{n}-\pi$, is classically treated  by using
techniques from functional analysis \citep[section 4.1.1]{gine+n:2008}. The variance term, $\hat{\pi}_{n}-\mathbb{E}\hat{\pi}_{n}$, is usually treated using
empirical process technique in the case of independent random variables. In
the next, we provide some results on the asymptotic behavior of the variance term.

We shall consider kernel functions $K:\mathbb{R}^{d}\rightarrow\mathbb{R}$
that taking one of the two following forms,
\begin{align}\label{assump:kernel}
(i)\quad K(x)=K^{(0)}(|x|),\qquad\text{or}\qquad(ii)\quad K(x)=\prod_{k=1}%
^{d}K^{(0)}(x_{k}),
\end{align}
where $K^{(0)} $ is a bounded function of bounded variation with support $[-1,1]$. From
\cite{nolan+p:1987}, the class of function
\[
\mathcal K = \{y\mapsto K((x-y)/h)\,:\,h>0,\,x\in\mathbb{R}\}\quad\text{is a uniformly
bounded VC class.}%
\]
This previous point has been used to handle the asymptotic analysis of kernel
estimate \citep{gigui2001} as well as in semiparametric problems as for instance in \cite{portier:2017}.

\begin{theorem}\label{th:unif_conv_kernel_1}Assume that the chain $X\subset \mathbb R^d$ satisfies the generic hypothesis (H) ,the stationary density $\pi$ is supposed to be bounded, the kernel $K $ is given by (\ref{assump:kernel}) and $K(x)\leq U$, for all $x\in \mathbb R^d$. Suppose that $h_n\to 0 $ and there exists $\beta>0$ such that $h_n\geq n^{-\beta}$.

\begin{enumerate}[(i)]
\item \label{th:first_point_kernel} If (PM) holds for $p>2$ and  $0<\beta (p/(p-1))<1/d$, we have
\[
\mathbb{E}_\nu \left[\sup_{x\in\mathbb{R}^{d}}|\hat{\pi}_n(x)- \mathbb E_\pi [\hat \pi_n(x)] |\right] = O\left( \sqrt
{\frac{\log\left(  nh_{n}^{-1}\right)  }{nh_{n}^{dp/(p-1)}}}\right) .
\]

\item \label{th:second_point_kernel} If (EM) holds and  $0<\beta <1/d$, we have
\[
\mathbb{E}_\nu\left[\sup_{x\in\mathbb{R}^{d}}|\hat{\pi}_n(x)-\mathbb E_\pi [\hat \pi_n(x)]  |\right] =  O\left( \sqrt
{\frac{\log(n)^2}{nh_{n}^{d}}}\right).
\]

\end{enumerate}
\end{theorem}

\begin{proof}

In virtue of Theorem \ref{th:expectation_bound}, it suffices to provide for both cases a sufficiently tight bound on $R_{n,B}(\mathcal K) $. First we consider (\ref{th:first_point_kernel}). By Jensen inequality we have
\[
\left(\frac{1}{\ell(B)}\sum_{X_{i}\in B}K((x-X_{i})/h_{n})\right)^{2}\leq\frac{1}{\ell (B)}%
\sum_{X_{i}\in B}K((x-X_{i})/h_{n})^{2}%
\]
and for any $\tilde{L}$,\ we
get\ by using\ the expression of Pitman's occupation measure
\begin{align*}
\sigma^{\prime 2} &=\mathbb E_A \left[ \left(\sum_{X_{i}\in B}K((x-X_{i})/h_{n})\right)^{2} \right]\\
&  \leq\mathbb E_A\left[  \ell(B)\sum_{X_{i}\in
B}K((x-X_{i})/h_{n})^{2}\right] \\
&  \leq\tilde{L}\mathbb E_A[\ell(B)]\mathbb{E}_{\pi}\left[  K((x-X)/h_{n}%
)^{2}\right]  +U^{2}\mathbb E_A\left[  \ell(B)^{2}\mathrm{1}_{\{\ell
(B)>\tilde{L}\}}\right] \\
&  \leq\tilde{L}\mathbb E_A[\tau_A]h_{n}^{d}\Vert\pi\Vert_{\infty}v_{K}
+\frac{U^{2}\mathbb E_A\left[ \tau_A^{p}\right]  }{\tilde{L}^{p-2}}.
\end{align*}
Equilibrating between the first and second term gives $\tilde{L}= A_1 h_{n}^{-d/(p-1)}$, $A_1>0$, which gives 
$\sigma^{\prime 2}\leq A_2 h_{n}^{d(p-2)/(p-1)}$, $A_2>0$. Applying Theorem \ref{th:rademacher_complexity}, we get
\begin{align}\label{bound:L}
&R_{n,B}(\mathcal K) \\
&\leq A_3 \left( \frac{L\log(C L UA_2 h_{n}^{-d(p-2)/(p-1)})}{nh_{n}^{d}}+\sqrt{\frac{ h_{n}
^{d(p-2)/(p-1)}\log(C L UA_2 h_{n}^{-d(p-2)/(p-1)} )}{nh_{n}^{2d}}}+\frac{1}%
{L^{p-1}h_{n}^{d}}\right),\nonumber\\
&= A_3 \left( \frac{L\log(C L UA_2 h_{n}^{-d(p-2)/(p-1)})}{nh_{n}^{d}}+\sqrt{\frac{ \log(C L UA_2 h_{n}^{-d(p-2)/(p-1)} )}{nh_{n}^{dp/(p-1)}}}+\frac{1}%
{L^{p-1}h_{n}^{d}}\right),\nonumber
\end{align}
Choose $L$ by equilibrating the first and last term of the
preceding decomposition
\[
L_{n}=A_4 \left(  \frac{n}{\log(n^{1/p}h_{n}^{-d(p-2)/(p-1)})}\right)  ^{1/p},
\]
with $A_4>0$. Let $\alpha_n =n^{1/p}h_{n}^{-d(p-2)/(p-1)}$. Then from (\ref{bound:L}) we get
\begin{align*}
R_{n,B}(\mathcal K) & \leq    A_3 \sqrt{\frac{\log( CUA_2 \alpha_n/ \log \alpha_n) }{nh_{n}^{dp/(p-1)}}}+2A_3A_4\left(
\frac{\log( CUA_2 \alpha_n/ \log \alpha_n) }{nh_{n}^{dp/(p-1)}}\right)
^{(p-1)/p} (\log(\alpha_n) )^{-1/p} \\
&  =A_3 \sqrt{\frac{\log(CUA_2 \alpha_n)}{nh_{n}^{dp/(p-1)}}%
}(1+o(1)),
\end{align*}
where the last equality is because, by assumptions on $h_n$, it holds that  $\alpha_n \leq n^{\alpha_1}$ and $nh_{n}^{dp/(p-1)} \geq n^{\alpha_2}$, for some positive constants $\alpha_1$ and $\alpha_2$. Comparing the previous bound on $R_{n,B}(\mathcal K)$ with the other terms given in Theorem \ref{th:expectation_bound} leads to the result.

In the second case (ii), a similar
bound is valid for $\sigma^{\prime 2}$, we have
\begin{align*}
\sigma^{\prime 2} &  \leq\tilde{L}\mathbb E_A[\ell(B)]h_{n}^{d}\Vert\pi
\Vert_{\infty}v_{K}+U^{2}\mathbb E_A\left[  \ell(B)^{2}\exp(\lambda
\ell(B)/2)\right]  \exp(-\lambda\tilde{L}/2)\\
&  \leq\tilde{L}\mathbb E_A[\tau_A]h_{n}^{d}\Vert\pi\Vert_{\infty}v_{K}%
+\frac{8U^{2}}{\lambda^{2}}\mathbb E_A\left[  \exp(\tau_A\lambda)\right]
\exp(-\lambda\tilde{L}/2)
\end{align*}
Taking $\tilde{L}=2\log(h_{n}^{-d})/\lambda$ gives $\sigma^{\prime 2} \leq  B_1
h_{n}^{d}\log(h_{n}^{-d})$, $B_1>0$. Then using (\ref{th:second_point}) in Theorem
\ref{th:rademacher_complexity}, we find
\[
R_{n,B}(\mathcal K) \leq B_2 \left(\frac{L\log(CLU B_1 h_{n}^{d}\log(h_{n}^{-d}) )}{nh_{n}^{d}}+\sqrt{\frac{\log(h_{n}^{-d}%
)\log(CLU h_{n}^{d}\log(h_{n}^{-d}))}{nh_{n}^{d}}}+\frac{\exp(-L\lambda/2)}{h_{n}^{d}}\right) ,
\]
with $B_2>0$. Choosing $L_{n}=2\log(n)/\lambda$ we finally get
\[
R_{n,B}(\mathcal K) \leq B_2 \sqrt{\frac{\log
(h_{n}^{-d})\log(\log(n) \log(h_n^{-d}) h_{n}^{-d})}{nh_{n}^{d}}} (1+o(1)) .
\]
As before $R_{n,B}(\mathcal K)$ is the leading term among the terms that appear in the bound of Theorem \ref{th:expectation_bound}.
\end{proof}

Comparing the rate of convergence given in Theorem
\ref{th:unif_conv_kernel_1} with usual rate of $ \sqrt{|\log(h_n)|/(nh_n^d) }$ corresponding to the independent case, we see that the rate of the Markovian
setting are slightly poorer.  Even when the regeneration time has exponential moments, a loss of a factor $\log(n)^{1/2}$ is observed with respect to the independent case. This loss is due to the variance term that scales differently due to the block size. To fill this gap, we provide in the following theorem an additional assumption on the chain $X$ that ensures the same rate as in the independent case.

\begin{theorem}
\label{th:unif_conv_kernel_2} 
Assume that the chain $X\subset \mathbb R^d$ satisfies the generic hypothesis (H) and (EM), the stationary density $\pi$ is supposed to be bounded, the kernel $K $ is given by (\ref{assump:kernel}) and $K(x)\leq U$, for all $x\in \mathbb R^d$. Suppose that $h_n\to 0 $ and that $ \sqrt{|\log(h_n)|/(nh_n^d) }\to 0$, if there exist $p>2$ and $C>0$ such that for all $x\in E$, $\pi(x )
\mathbb{E}_{x} [\tau_{A}^{p}]\leq C$, then we have
\begin{align*}
\mathbb{E}_\nu\left[ \sup_{x\in\mathbb{R}^{d} }|\hat{\pi}_n(x)-\mathbb E_\pi [\hat \pi_n(x)]  | \right] = O\left(\sqrt
{\frac{ |\log( h_{n})| }{ n h_{n}^{d}} }\right) .
\end{align*}

\end{theorem}

\begin{proof}
The main step is to show that there exists $c>0$ such that
\begin{align}
\label{cond:same_rate}\mathbb E _A \left[   \left( \sum_{i=1 }^{\tau_A} K((x-X_{i})/h_{n}) \right)^2\right]  \leq c h_{n}^{d},\qquad \text{for all } x\in E,
\end{align}
then the conclusion will follow straightforwardly. The fact that (\ref{cond:same_rate}) holds true follows from Lemma 11 in \cite{azais:2016}, which gives that, for any measurable function $f$,
\begin{align*}
\mathbb E _A \left[   \left( \sum_{i=1 }^{\tau_A}  f(X_i)\right)^2\right]  \leq C_1 ( \pi(f^2)  +\mathbb E_A [f(X_0)^2\tau_A^{p}]) ,
\end{align*}
with $C_1>0$. Whenever $f (X) = K((x-X)/h_n)$, we get that $\pi(f^2)\leq v_K \|\pi\|_\infty h_n^d $, where $\|\pi\|_\infty = \sup_{x\in \mathbb R^d}|\pi(x)|$, and defining $g(x) =  \pi(x ) \mathbb E_x [\tau_A^p]$ we get
\begin{align*}
\mathbb E_A[f(X_0)^2\tau_A^{p}]  = \int f(y ) g(y) \,d y \leq C v_K  h_n^d.
\end{align*}
Hence we have obtained (\ref{cond:same_rate}). It follows from Theorem \ref{th:rademacher_complexity} that
\begin{align}\label{bound:L_2}
R_{n,B}(\mathcal K)\leq C_3 \left(\frac{ L \log( CLU c^{-1} h_n^{-d})}{nh_n^d} +  \sqrt {\frac{\log( CLU c^{-1}h_n^{-d}) }{ n h_n^{d}} }   +  \frac{ \exp(-L\lambda/2) }{ h_n^d }  \right).
\end{align}
Setting $L_n = 2 \log(n) /\lambda$ we obtain the desired result by applying Theorem \ref{th:expectation_bound}.
\end{proof}

\begin{remark}[on the bandwidth]
In the independent case, given $x\in \mathbb R$, the variance of $\hat f(x)$ is ensured to vanish whenever $nh_n^d\to +\infty$, and asking for $nh_n^d/|\log(h_n)| \to +\infty $ is a slight additional requirement to ensure that convergence happens uniformly over $\mathbb R$. In Theorem \ref{th:unif_conv_kernel_2}, the assumptions on the bandwidth are the same as in the independent case. In Theorem \ref{th:unif_conv_kernel_1}, the fact that $h_n\geq n^{-\beta}$ is slightly stronger.
\end{remark}


\subsection{Metropolis-Hasting algorithm}\label{sec:applications_2}

Bayesian estimation requires to compute moments of the so called \textit{posterior distribution} whose probability density function $\pi$ is given by
\begin{align*}
 \pi(\theta) = \frac{ \mathcal L (\theta)}{\int  \mathcal L (\theta)d \theta }\qquad \theta\in \mathbb R^d, 
\end{align*}
where $ \mathcal L $ is a positive function which stands for the likelihood of the observed data. The quantities of interest writes as $  \int gd \pi$, for some given measurable functions $g:\mathbb R^d\to \mathbb R$, and are unfortunately unknown. A particular feature in this framework is that the integral at the denominator of $\pi$ is unknown and difficult to compute making impossible to generate observations directly from $\pi$. Markov Chains Monte Carlo (MCMC) methods aim to produce samples $X_1,\ldots,X_n$ in $\mathbb R^d$ that are approximately distributed according to $\pi$. Then $\int gd \pi$ is classically approximated by the empirical average over the chain :
\begin{align*}
 n^{-1} \sum_{i=1} ^n g(X_i).
\end{align*}
For inference, Bayesian credible intervals are usually computed using the quantiles the coordinate chains (see below). We refer to \cite{robert:2004} for a complete description of MCMC methods. In what follows, we focus on the special MCMC method called Metropolis-Hasting (MH). Aim is to derive new concentration inequalities for suprema of $\sum_{i=1} ^n g(X_i)$ over $g$ in some VC classes. We show in particular that our results provide convergence rates for the estimation of Bayesian credible intervals. 

Let us introduce the MH algorithm with target density $\pi : \mathbb R^d \to \mathbb R_{\geq 0}$ and proposal $Q(x,dy) = q(x,y) dy$, where $q$ is a positive function defined on $\mathbb R^d\times \mathbb R^d$ satisfying $\int q(x,y)dy = 1$. Define for any $(x,y)\in \mathbb R^d\times \mathbb R^d$,
\begin{align*}
\rho (x,y) = \left\{\begin{array}{ll}
\min \left(1, \frac{\pi(y)q(y,x)}{\pi(x)q(x,y)} \right)&\quad \text{if } \pi(x)q(x,y) >0, \\
1 &\quad \text{if } \pi(x)q(x,y) =0.
\end{array}\right. 
\end{align*}
The MH chain moves from $X_n$ to $X_{n+1}$ according to the following rule : 
\begin{enumerate}[(i)]
\item Generate 
\begin{align*}
&Y\sim Q(X_n, dy)\qquad \text{and} \qquad B\sim\mathcal B (\rho(X_n, Y)).
\end{align*}
\item Set
\begin{align*}
X_{n+1} = \left\{\begin{array}{lll} Y \quad &\text{if} \quad B = 1, \\
X_{n} \quad &\text{if} \quad B = 0.
\end{array} \right.
\end{align*}
\end{enumerate}
In the particular case that $q(x,y) = q(x-y) $, the algorithm is refereed to as the random walk MH. We call the chain $X$ the random walk MH chain. 

The asymptotic behavior  of the random walk MH algorithm has been studied in \cite{roberts+t:1996,jarner+h:2000} where central limit theorems are established based on the geometric ergodicity of the chain. From Remark \ref{remark:geo_ergo}, the results in \cite{roberts+t:1996,jarner+h:2000} imply that (EM) is satisfied. This allows to apply Theorem \ref{th:rademacher_complexity} almost directly for the random walk MH. For the sake of completeness, we provide the following alternative development, in which we verify (EM) via the (uniform) Doeblin condition. Contrary to \cite{roberts+t:1996,jarner+h:2000}, we focus on $\pi$ with bounded support. Hence we start by giving (in the next coming theorem) a condition implying the Doeblin condition. This condition will be easily verified for the random walk MH chain.

Denote by $B(x,\epsilon)$ the ball with centre $x$ and radius $\epsilon$ with respect to the Euclidean distance $\|\cdot\|$.


\begin{proposition}\label{prop:unif_doeblin}
Let $P$ be a transition kernel. Let $\Phi$ be a positive measure on $(\mathbb R^d,\mathcal B(\mathbb R^d))$. Suppose that $E=\supp(\Phi)$ is bounded and convex with nonempty interior. Suppose that there exists $\epsilon>0$ such that $\forall x\in E$, $P(x,dy) \geq \mathrm 1_{ B(x,\epsilon)}(y) \Phi(dy) $. Then there exists $C>0$ and $n\geq 1$ such that for any $x\in E$ and any measurable set $A\subset E$,
\begin{align}\label{eq_unif_doeblin}
P^n(x,A)\geq  C \Phi(A).
\end{align}

\end{proposition}

\begin{proof}

We decompose the proof according to $4$ steps.

\noindent\textit{First step:}
Let $0<\gamma\leq \eta$. There exists $c>0$ such that for any 
 $(x,y) \in E$, it holds that
\begin{align}\label{eq:first_step}
\int  \mathrm 1_{ B(x,\eta)   }(x_1) \mathrm 1_{ B(y, \gamma)  } (x_1) \,\Phi(d x_1)    \geq c \mathrm 1_{  B(x, \eta + \gamma/4)  }(y)  .
\end{align}
To obtain the previous statement, we can restrict our attention to the case when $ y \in B(x, \eta + \gamma/4)  $. Else the inequality is trivial. Note that there exists a point $m$ lying strictly in the line segment between $x$ and $y$ such that
\begin{align*}
B ( m, \gamma/4   ) \subset \{  B(x,\eta ) \cap  B(y,\gamma )\}.
\end{align*}
By convexity of $E$, $m\in E$. Hence
\begin{align*}
\int  \mathrm 1_{ B(x,\eta)   }(x_1) \mathrm 1_{ B(y, \gamma)  } (x_1) \Phi(d x_1)     &\geq \Phi \{B ( m, \gamma/4   )  \} \geq \inf_{m\in E} \Phi \{ B ( m, \gamma/4   )   \}.
\end{align*}
But he function $m\mapsto \Phi (B ( m, \gamma/4   )\cap E )$ is continuous on $E$ and positive for each $m\in E$, by definition of the support and the fact that $m$ is an interior point of $E$ by convexity.

\noindent\textit{Second step:} We iterate (\ref{eq:first_step}) to obtain the following statement.
For any $n\geq 1$, there exists $C_n>0$ such that for any $(x,y) \in E$, it holds that
\begin{align*}
&\int \ldots \int \mathrm 1_{B(x,\epsilon)   }(x_1) \mathrm 1_{B(x_2,\epsilon)   }(x_1) \ldots    \mathrm 1_{B(x_n,\epsilon)   }(x_{n-1})     \mathrm 1_{B(y,\epsilon)   }(x_n) \,  \Phi(d x_1) \ldots \Phi(d x_n)  \\
&\geq C_n \mathrm 1_{B(x , \epsilon (1 + n/4))    }(y).
\end{align*}

\noindent\textit{Third step:} Take $n$ such that $ \epsilon (1 + n/4)  > \sup_{(x,y)\in E} \|x-y\| $. Then for any $x\in E$ and $y\in E$, $ y\in B(x , \epsilon (1 + n/4))     $. It follows that there exists $C_n>0$ such that for all $(x,y) \in E$,
\begin{align*}
&\int \ldots \int \mathrm 1_{B(x,\epsilon)   }(x_1) \mathrm 1_{B(x_2,\epsilon)   }(x_1) \ldots    \mathrm 1_{B(x_n,\epsilon)   }(x_{n-1})     \mathrm 1_{B(y,\epsilon)   }(x_n) \,  \Phi(d x_1) \ldots \Phi(d x_n) \geq C_n.
\end{align*}

\noindent\textit{Fourth step:}
Using the last step and the assumption on $P$, it holds that for any $x\in E$ and any measurable set $A\subset E$,
\begin{align*}
&P^n(x,A)\\
&\geq \int \mathrm 1_{B(x,\epsilon)   }(x_1)  \ldots    \mathrm 1_{B(x_n,\epsilon)   }(x_{n-1})     \mathrm 1_{B(y,\epsilon)   }(x_n) \mathrm 1_{\{ y\in A   \}  } \,  \Phi(d x_1) \ldots \Phi(d x_n)  \Phi(d y) \\
&\geq C_n \Phi(A).
\end{align*}

\end{proof}

Applying the previous result to the random walk MH, we obtain that the random walk MH verifies (EM). This is the main conclusion of the following statement.

\begin{proposition}\label{prop:exp_moment_MH}
Let $\pi  $ be a bounded probability density supported by $E\subset \mathbb R^d$, a bounded and convex set with non-empty interior.  Suppose that there exists $b>0$ such that $\forall (x,y) \in \mathbb R^d\times \mathbb R^d$, $ q(x,y) \geq b \mathrm 1_{B(x, \epsilon)   }(y) $. Then the MH chain verifies ({H}) and (EM).
\end{proposition}
\begin{proof}

Because $\rho(x,y)\geq \pi(y)/\|\pi\|_\infty$, the Markov kernel $P$ of the MH chain verifies, for any $x\in E$,
\begin{align}\label{eq:mh_key}
P(x,dy) \geq \rho(x,y) Q(x,d y) \geq \|\pi\|_\infty^{-1}  \mathrm 1_{ B(x ,\epsilon)   }(y) \pi(y) d y.
\end{align}

Let $z\in E$. From (\ref{eq:mh_key}), whenever $x\in  B(z,\epsilon/2)$ and $A\in \mathcal B(\mathbb R^d) $,  
\begin{align*}
P(x,A)\geq \|\pi\|_\infty^{-1}  \pi(A\cap B(z,\epsilon/2) ) .
\end{align*}
This means that any ball with positive radius is $\pi_{| B(z,\epsilon/2)}$-small with $m=1$. Following \cite[proof of Theorem 2.2]{roberts+t:1996}, this implies the aperiodicity of the chain.

Applying Proposition \ref{prop:unif_doeblin} with $\Phi(dy) = \|\pi\|_\infty^{-1} \pi(y) d y$, we deduce that whenever $\pi(A)>0$, there exists $n\geq 1$ such that $P^n(x,A)>0$. This is $\pi$-irreducibly.

Applying Proposition \ref{prop:unif_doeblin} with $\Phi(dy) = \|\pi\|_\infty^{-1} \pi(y) d y$, we obtain (\ref{eq_unif_doeblin}) which implies (EM) in virtue of Theorem 16.0.2 in \cite{meyn+t:2009}. More precisely, in their Theorem 16.0.2, (\ref{eq_unif_doeblin}) implies point (iv) which is equivalent to point (vii). That is, we have shown that whenever $\psi(B)>0$, there is $\lambda_B>0 $ such that $\sup_{x\in E} \mathbb E_x [\exp(\lambda_B\tau_B)]<\infty$. This is stronger than positive Harris recurrence. The latter is true for any atom $A$ of the extended chain as well. 

\end{proof}

Based on Proposition \ref{prop:exp_moment_MH}, we are in position to apply point (\ref{th:second_point}) of Proposition \ref{th:rademacher_complexity} to the random walk MH.

\begin{proposition}\label{bound:vc_class_mh}
Let $\mathcal{G}$ be a countable VC class of measurable functions on
$S$ bounded by $U$ with characteristics $(C,v)$. Let $\pi  $ be a bounded probability density supported by $E\subset \mathbb R^d$, a bounded and convex set with nonempty interior.
Suppose that there exists $b>0$ such that $\forall (x,y) \in \mathbb R^d\times \mathbb R^d$, $ q(x,y) \geq b \mathrm 1_{  B(x,\epsilon)     }(y) $. Then, for all $n\geq 1$, it holds
\begin{align*}
&\mathbb E \left[   \sup_{g\in \mathcal G} \left|  \sum_{i=1}^n \left(g(X_i ) - \pi (g)   \right)\right|  \right]\leq  D \sqrt{ n (1\vee \log(  \log(n)))} ,
\end{align*}
where $D$ depends only on $v,C, U$ and on the tails of the regeneration time. Moreover, for any $n/\log(n)\geq {\tau^{3}} /   (U \mathbb{E}_{A}[\tau_{A}^2])$, we have with probability $1-\delta$,
\begin{align*}
&\sup_{g\in \mathcal G} \left| \sum_{i=1}^n \left( g(X_i) - \pi(g)    \right)\right|  \leq \\
&KD \sqrt{ n  (1\vee \log(\log(n)))}  + \max\left(   \sqrt{ nU^2 \mathbb E_A[\tau _A^{2}  ] K\log\left(\frac
{K}{\delta}\right)},\log\left(\frac{K}{\delta}\right)\frac{\tau^{3}U\log(n)}{\mathbb{E}_{A}[\tau_{A}
]}\right),
\end{align*}
where $K>0$ is a universal constant.
\end{proposition}

\begin{proof}
Set $\sigma^{\prime 2} =  U^2 \mathbb E_A[\tau _A^{2}  ] $ and apply Theorem \ref{th:rademacher_complexity} to get that 
\[
R_{n,B}(\mathcal{F}^{\prime})\leq M\left[  v LU\log\frac{CL }{ \mathbb E_A[\tau _A^2  ] ^{1/2}}+\sqrt{v n  U^2 \mathbb E_A[\tau _A^2  ]  \log\frac{CL }{ \mathbb E_A[\tau _A^2  ] ^{1/2}}}\right]  +nU\exp(-L\lambda/2)C_{\lambda}.
\]
Take $L=2\log(n)/\lambda$ to obtain
\begin{align*}
R_{n,B}(\mathcal{F}^{\prime})& \leq M\left[  2\log(n) v U\log(A\log(n))/\lambda +\sqrt{v n U^2 \mathbb E_A[\tau _A^2  ]  \log(A\log(n))
}\right]  +U C_{\lambda}
\end{align*}
with $ A = 2 C/ (\lambda \mathbb E _A [\tau _A^2 ]^{1/2}) $. We obtain the first stated result by straightforward manipulations. The second result is a direct consequence of Theorem \ref{th:proba_bound}.
\end{proof}

Let $k\in \{1,\ldots, d\}$ and denote by $X_i^{(k)}$ the $k$-th coordinate of $X_i$. Define the associated empirical cumulative distribution function 
for any $t\in \mathbb R $,
\begin{align*}
&\hat \Pi_k (t) =n^{-1} \sum_{i=1} ^n \mathrm 1 _{\{  X_i^{(k)} \leq t\}}.
\end{align*}
and the quantile function, for any $u\in (0,1)$,
\begin{align*}
&\hat Q_k(u) = \inf\{x\in \mathbb R \,:\, \hat F_k(x) \geq u \}.
\end{align*}
As a corollary of the previous result, we obtain an upper bound for the estimation error of Bayesian credible intervals  defined as $[\hat Q_k(u),\hat Q_k(1-u)]$.
The targeted interval is $[ Q_k(u), Q_k(1-u)]$, where $Q_k$ is the true quantile function of the posterior marginal distribution $\Pi_k$ whose associated density is denoted $\pi_k$.  

\begin{proposition}
Let $\pi  $ be a bounded probability density supported by $E\subset \mathbb R^d$, a bounded and convex set with nonempty interior.
Suppose that there exists $b>0$ such that $\forall (x,y) \in \mathbb R^d\times \mathbb R^d$, $ q(x,y) \geq b \mathrm 1_{  B(x,\epsilon)    }(y) $. For any $0< \gamma <1/4$, $1\leq k\leq d$, and any $n/\log(n)\geq {\tau^{3}} / \mathbb{E}_{A}[\tau_{A}^2]  $, we have with probability $1-\delta$,
\begin{align*}
&b_{k,\gamma} \sup_{u\in [2\gamma,1-2\gamma]} \left|  \hat Q_k (u) - Q_k(u)   \right|   \leq \\
&KD \sqrt{ n   (1\vee \log(\log(n))) }  + \max\left( \sqrt{ n \mathbb{E}_{A}[\tau_{A}^2]   K\log\left(\frac
{K}{\delta}\right)},\log\left(\frac{K}{\delta}\right)\frac{\tau^{3}\log(n)}{\mathbb{E}_{A}[\tau_{A}%
]}\right)
\end{align*}
where $b_{k,\gamma} = \inf_{u\in [\gamma, 1-\gamma] } \pi_k(Q_k(u)) $ and $D$ depends only on $v,C, U$ and the tails of the regeneration time. In particular, provided that $b_{k,\gamma} >0$,
\begin{align*}
 \sup_{u\in [2\gamma,1-2\gamma]} \left|  \hat Q_k(u) - Q_k(u)   \right|   =  O_{\mathbb P } \left( \sqrt {\frac{\log\log n }{n} } \right).
\end{align*}
\end{proposition}
\begin{proof}
Start by recalling the classical result; see e.g., \cite[Lemma 12]{portier:2017}; that whenever $F$ and $G$ are two cumulative distribution functions satisfying that $\sup_{t\in \mathbb R} |F(t)-G(t)| \leq \gamma$, then for any $u\in [\gamma,1- \gamma]$,
\begin{align*}
|F^- (u) - G^{-} (u) | &\leq \sup_{|\delta|\leq \gamma}  | G^-(u+\delta) - G^-(u) |,
\end{align*}
where $F^-$ and $G^-$ stands for the generalized inverse of $F$ and $G$. 
Consequently, if $G$ has a strictly positive density on $[F^-(\gamma), F^-(1-\gamma) ]$, we get that, for any $u\in [2\gamma,1- 2\gamma]$,
\begin{align*}
|F^- (u) - G^{-} (u) | &\leq \gamma \sup_{ u \in[\gamma, 1-\gamma] } \frac{\partial}{\partial u} G^{-}(u) = \gamma  \left(  \inf_{u\in [\gamma, 1-\gamma] } f(F^{-}(u)) \right)^{-1} .
\end{align*}
The previous is applied to $F=\hat  \Pi_k$, $G = \Pi_k$ and $\gamma$ taken from the second bound in Proposition \ref{bound:vc_class_mh}. In virtue of Example 2.5.4 in \cite{wellner1996}, we have that
\begin{align*}
\mathcal N ( \epsilon , \{\mathrm 1 _{(-\infty, t]} \,:\, t\in \mathbb R \} , \|\cdot\| _{L_2(Q)})  \leq 2\epsilon^{-2},
\end{align*}
which allows to apply Proposition \ref{bound:vc_class_mh} with $U=1$, $C = v=2$. 

\end{proof}

\begin{remark}
In contrast with the study of kernel density estimator given in section \ref{sec:applications_1}, the approach taken in this section cannot take advantage of classes with small variance. In particular, in Theorem \ref{bound:vc_class_mh}, the variance over the class $\mathcal G$ is crudely bounded by $U\mathbb{E}_{A}[\tau_{A}^2] $. This does not generate any loss for the final application to credible intervals because the underlying class functions is  $\{\mathrm 1 _{(-\infty, t]} \,:\, t\in \mathbb R \}$ with variance $\Pi(t)(1-\Pi(t))$ whose maximum is $1/4$.
\end{remark}

\paragraph{Acknowledgment} The authors are grateful to Gabriela Ciolek for helpful comments on an earlier version of the paper.

\bibliographystyle{chicago}
\bibliography{concentration_9.bbl}

\end{document}